\documentclass[12pt, a4paper]{article}

\usepackage{indentfirst}
\usepackage[leqno]{amsmath}
\usepackage{amssymb}
\usepackage{amsthm}
\usepackage{enumerate}
\usepackage{cases}
\usepackage[center]{titlesec}
\usepackage{fancyhdr}
\usepackage[colorlinks, 
linkcolor=blue,
anchorcolor=blue, 
citecolor=red]{hyperref}

\numberwithin{equation}{section} 
\newtheorem{lem}{Lemma}[section]
\newtheorem{thm}[lem]{Theorem}

\newtheorem{rmk}[lem]{Remark}

\newtheorem{cor}[lem]{Corollary}

\title{Sobolev Inequalities In Manifolds With Asymptotically Nonnegative Curvature\thanks{Supported by NSFC Grants No. 11771087, No. 12171091 and No. 11831005.}}
\date{}

\begin{document}
	\author{Yuxin Dong, Hezi Lin, Lingen Lu}
	\maketitle

	\begin{abstract}
		ABSTRACT. Using the ABP-method as in a recent work by Brendle \cite{Br2}, we establish some sharp Sobolev and isoperimetric inequalities for compact domains and submanifolds in a complete Riemannian manifold with asymptotically nonnegative curvature.  These inequalities generalize those given by Brendle in the case of complete Riemannian manifolds with nonnegative curvature.
	\end{abstract}

\section{Introduction}
It is known that Sobolev inequalities, as an important analytic tool in geometric analysis, have close connections with isoperimetric inequalities. The classical isoperimetric inequality for a bounded domain $D$ in $\mathbb{R}^n$ says that
$$
n^n|B^n||D|^{n-1}\leq|\partial D|^n  
$$
where $B^n$ denotes the unit ball in $\mathbb{R}^n$, and the equality holds if and only if $D$ is a ball. There have been numerous works generalizing this inequality to different settings (cf. \cite{Os}, \cite{Cha}, \cite{Choe}). 

The isoperimetric inequalities on minimal surfaces or minimal submanifolds have a long history. For example, \cite{Car}, \cite{Cha}, \cite{Hs}, \cite{Re}, \cite{LSY}, \cite{St}, \cite{To} investigated the isoperimetric inequality on minimal surfaces under various conditions, while the famous Michael-Simon Sobolev inequality for general dimensions (\cite{Al}, \cite{MS}) implies an isoperimetric inequality for minimal submanifolds, but with a non-sharp constant.
It is conjectured that any $n$-dimensional minimal submanifold $\Omega$ of $\mathbb{R}^N$ satisfies the classical isoperimetric inequality: $n^n|B^n||\Omega|^{n-1}\leq|\partial\Omega|^n$ with equality holds if and only if $\Omega$ is a ball in an $n$-plane of $\mathbb{R}^N$.
Recently, S. Brendle \cite{Br1}, inspired by the ABP method as in \cite{Ca} and \cite{Tr}, established a Michael-Simon-Sobolev type inequality on submanifolds of arbitrary dimension and codimension, which is sharp if the codimension is at most 2. 
In particular, his result implies a sharp isoperimetric inequality for minimal submanifolds in Euclidean space of codimension at most 2.
Later, Brendle \cite{Br2} also generalized his results in \cite{Br1} to the case that the ambient space is a Riemannian manifold with nonnegative curvature.
In \cite{Jo}, F. Johne gave a sharp Sobolev inequality for manifolds with nonnegative Bakry-\'Emery Ricci curvature, which generalizes Brendle's results in \cite{Br2}. In \cite{BK}, Balogh and Kris\'aly proved a sharp isoperimetric inequality in metric measure spaces satisfying $\mathsf{CD}(0,N)$ condition which implies the sharp isoperimetric inequalities in \cite{Br2} and \cite{Jo}. Moreover, they also obtained a sharp $L^p$-Sobolev inequality for $p\in(1,n)$ on manifolds with nonnegative Ricci curvature and Euclidean volume growth. In a recent preprint \cite{APEP}, the authors also investigated sharp and rigid isoperimetric comparison theorems in $\mathsf{RCD}(K,N)$ metric measure spaces.

In this paper, we generalize Brendle's results in \cite{Br2} to the case that the ambient space has asymptotically nonnegative curvature. The notion of asymptotically nonnegative curvature was first introduced by U. Abresch \cite{Ab1}. Some important geometric, topological and analysis problems have been investigated for this kind of manifolds (cf. \cite{Ab2}, \cite{AG}, \cite{Ka1,Ka2}, \cite{LT1,LT2}, \cite{Zhu}, \cite{GPZ}, \cite{Ba1}, \cite{Zha}, etc). Now we recall its definition as follows. Let $\lambda:[0,+\infty)\to[0,+\infty)$ be a nonnegative and nonincreasing continuous function satisfying
\begin{equation}\label{b0}
b_0:=\int_0^{+\infty}s\lambda(s)ds<+\infty,
\end{equation}
which implies
\begin{equation}\label{b1}
	b_1:=\int_0^{+\infty}\lambda(s)ds<+\infty.
\end{equation}
A complete noncompact Riemannian manifold $(M,g)$ of dimension $n$ is said to have asymptotically nonnegative Ricci curvature (resp. sectional curvature) if there is a base point $o\in M$ such that
\begin{equation}\label{asy-cur}
\mathrm{Ric}_q(\cdot,\cdot)\geq-(n-1)\lambda(d(o,q))g
\quad (resp.\ \mathrm{Sec}_q\geq-\lambda(d(o,q))),
\end{equation}
where $d(o,q)$ is the distance function of $M$ relative to $o$. Clearly, this notion includes the manifolds whose Ricci (resp. sectional) curvature is either nonnegative outside a compact set or asymptotically flat at infinity. In particular, if $\lambda\equiv0$ in \eqref{asy-cur}, then this becomes the case treated in \cite{Br2}.

Let $h(t)$ be the unique solution of 
\begin{equation}\label{h}
\left\{
\begin{aligned}
	&h''(t)=\lambda(t)h(t),\\
	&h(0)=0, h'(0)=1.
\end{aligned}
\right.
\end{equation}
By ODE theory, the solution $h(t)$ of \eqref{h} exists for all $t\in[0,+\infty)$. According to \cite{Zhu} (see also Theorem 2.14 in \cite{PRS}), the function
$$
\frac{|\{q\in M:d(o,q)<r\}|}{n|B^n|\int_0^rh^{n-1}(t)dt}
$$
is a non-increasing function on $[0,+\infty)$ and thus we may introduce the asymptotic volume ratio of $M$ by
\begin{equation}\label{avr}
	\theta:=\lim_{r\to+\infty}\frac{|\{q\in M:d(o,q)<r\}|}{n|B^n|\int_0^rh^{n-1}(t)dt},
\end{equation}
with $\theta\leq1$. In particular, we have $|\{q\in M:d(o,q)<r\}|\leq|B^n|e^{(n-1)b_0}r^n$.

First, by combining the method in \cite{Br2} with some comparison theorems, we establish a Sobolev type inequality for a compact domain in a Riemannian manifold with asymptotically nonnegative Ricci curvature as follows.

\begin{thm}\label{thm1.1}
	Let $M$ be a complete noncompact $n$-dimensional manifold of asymptotically nonnegative Ricci curvature with respect to a base point $o\in M$. Let $\Omega$ be a compact domain in $M$ with boundary $\partial\Omega$, and let $f$ be a positive smooth function on $\Omega$. Then
	$$
	\int_{\partial\Omega} f+\int_\Omega |D f|+
	2(n-1)b_1\int_\Omega f
	\geq n|B^n|^{\frac {1}{n}}\theta^{\frac1n}
	\Big(\frac{1+b_0}{e^{2r_0b_1+b_0}}\Big)^\frac{n-1}{n}
	\Big(\int_\Omega f^\frac{n}{n-1}\Big)^\frac{n-1}{n},
	$$
	where $r_0=\max\{{d}(o,x)|x\in\Omega\}$, $\theta$ is the asymptotic volume ratio of $M$ given by \eqref{avr} and $b_0,b_1$ are defined in \eqref{b0} and \eqref{b1}.
\end{thm}	

The following result characterizes the case of equality in Theorem \ref{thm1.1}:

\begin{thm}\label{thm1.2}
	Let $M$ be a complete noncompact $n$-dimensional manifold of asymptotically nonnegative Ricci curvature with respect to a base point $o\in M$. Let $\Omega$ be a compact domain in $M$ with boundary $\partial\Omega$, and let $f$ be a positive smooth function on $\Omega$.  If
	$$
	\int_{\partial\Omega} f+\int_\Omega |D f|+
	2(n-1)b_1\int_\Omega f
	=
	n|B^n|^{\frac {1}{n}}\theta^{\frac1n}
	\Big(\frac{1+b_0}{e^{2r_0b_1+b_0}}\Big)^\frac{n-1}{n}
	\Big(\int_\Omega f^\frac{n}{n-1}\Big)^\frac{n-1}{n},
	$$
	where $r_0=\max\{{d}(o,x)|x\in\Omega\}$, $\theta$ is the asymptotic volume ratio of $M$ given by \eqref{avr} and $b_0,b_1$ are defined in \eqref{b0} and \eqref{b1}.
	Then $b_0=b_1=0$, $M$ is isometric to Euclidean space, $\Omega$ is a ball, and $f$ is constant.
\end{thm}

Taking $f=1$ in Theorem \ref{thm1.1}, we obtain a sharp isoperimetric inequality:

\begin{cor}\label{cor1.3}
	Let $M,\Omega, r_0,\theta,b_0,b_1$ be as in Theorem \ref{thm1.1}. Then
	$$
	|\partial \Omega|\geq \Big(n|B^n|^{\frac{1}{n}}
	\theta^{\frac{1}{n}}
	\Big(\frac{1+b_0}{e^{2r_0b_1+b_0}}\Big)^\frac{n-1}{n}
	-
	2(n-1)b_1|\Omega|^{\frac{1}{n}}\Big)|\Omega|^{\frac{n-1}{n}}.
	$$
	Furthermore, the equality holds if and only if $M$ is isometric to Euclidean space and $\Omega$ is a ball.
\end{cor}


\begin{rmk}
	If $M$ has nonnegative Ricci curvature, then $b_0=b_1=0$ and Corollary \ref{cor1.3} becomes
	$$
	|\partial\Omega|\geq n|B^n|^{\frac1n}\theta^{\frac1n},
	$$
	which was first given by V. Agostiniani, M. Fogagnolo, and L. Mazziari \cite{AFM} in dimension 3 and obtained by S. Brendle \cite{Br2} for any dimension, see also \cite{ML} for related results in $3\leq n\leq 7$.
\end{rmk}

Similarly, we may establish a Sobolev type inequality for a compact submanifold (possibly with boundary) in a Riemannian manifold with asymptotically nonnegative sectional curvature as follows.

\begin{thm}\label{thm1.4}
	Let $M$ be a complete noncompact $(n+p)$-dimensional manifold of asymptotically nonnegative sectional curvature with respect to a base point $o\in M$. Let $\Sigma$ be a compact $n$-dimensional submanifold of $M$ (possibly with boundary $\partial\Sigma$), and let $f$ be a positive smooth function on $\Sigma$. If $p\geq2$, then
	$$
	\begin{aligned}
		&\int_{\partial\Sigma}f+\int_\Sigma\sqrt{|D^\Sigma f|^2+f^2|H|^2}
		+2nb_1\int_\Sigma f\\
		&\geq
		n\Big(\frac{(n+p)|B^{n+p}|}{p|B^p|}\Big)^{\frac{1}{n}}\theta^{\frac{1}{n}}
		\Big(\frac{1+b_0}{e^{2r_0b_1+b_0}}\Big)^{\frac{n+p-1}{n}}
		\Big(\int_\Sigma f^{\frac{n}{n-1}}\Big)^{\frac{n-1}{n}},
	\end{aligned}
	$$
	where $r_0=\max\{d(o,x)|x\in\Sigma\}$, $H$ is the mean curvature vector of $\Sigma$, $\theta$ is the asymptotic volume ratio of $M$ given by \eqref{avr} and $b_0,b_1$ are defined in \eqref{b0} and \eqref{b1}. 
\end{thm}

Note that $(n+2)|B^{n+2}|=2|B^2||B^n|$. Hence, we obtain the following Sobolev type inequality for codimension 2:

\begin{cor}\label{cor1.5}
	Let $M$ be a complete noncompact $(n+2)$-dimensional manifold of asymptotically nonnegative sectional curvature with respect to a base point $o\in M$. Let $\Sigma$ be a compact $n$-dimensional submanifold of $M$ (possibly with boundary $\partial\Sigma$), and let $f$ be a positive smooth function on $\Sigma$. Then
	$$
	\begin{aligned}
		&\int_{\partial\Sigma}f+\int_\Sigma\sqrt{|D^\Sigma f|^2+f^2|H|^2}
		+2nb_1\int_\Sigma f\\
		&\geq
		n|B^n|^{\frac{1}{n}}\theta^{\frac{1}{n}}
		\Big(\frac{1+b_0}{e^{2r_0b_1+b_0}}\Big)^{\frac{n+1}{n}}
		\Big(\int_\Sigma f^{\frac{n}{n-1}}\Big)^{\frac{n-1}{n}},
	\end{aligned}
	$$
	where $r_0=\max\{d(o,x)|x\in\Sigma\}$, $H$ is the mean curvature vector of $\Sigma$, $\theta$ is the asymptotic volume ratio of $M$ given by \eqref{avr} and $b_0,b_1$ are defined in \eqref{b0} and \eqref{b1}.
\end{cor}

The following result characterizes the case of equality in Corollary \ref{cor1.5}:

\begin{thm}\label{thm1.6}
	Let $M, \Sigma, f, r_0, H, \theta, b_0, b_1$ as in Corollary \ref{cor1.5}. If
	$$
	\begin{aligned}
		&\int_{\partial\Sigma}f+\int_\Sigma\sqrt{|D f|^2+f^2|H|^2}
		+2nb_1\int_\Sigma f\\
		&=
		n|B^n|^{\frac{1}{n}}\theta^{\frac{1}{n}}
		\Big(\frac{1+b_0}{e^{2r_0b_1+b_0}}\Big)^{\frac{n+1}{n}}
		\Big(\int_\Sigma f^{\frac{n}{n-1}}\Big)^{\frac{n-1}{n}}.
	\end{aligned}
	$$
	Then $b_0=b_1=0$ and $M$ is isometric to Euclidean space, $\Sigma$ is a flat ball, and $f$ is constant.
\end{thm}

Letting $f=1$ in Corollary \ref{cor1.5}, we obtain a sharp isoperimetric inequality for minimal submanifolds of codimension 2 as follows.

\begin{cor}
	Let $M$ be a complete noncompact $(n+2)$-dimensional manifold of asymptotically nonnegative sectional curvature with respect to a base point $o\in M$. Let $\Sigma$ be a compact $n$-dimensional mininal submanifold of $M$ (possibly with boundary $\partial\Sigma$). Then
	$$
	|\partial\Sigma|\geq n\Big(|B^n|^{\frac{1}{n}}\theta^{\frac{1}{n}}
	\big(\frac{1+b_0}{e^{2r_0b_1+b_0}}\big)^{\frac{n+1}{n}}-2b_1|\Sigma|^{\frac{1}{n}}\Big)|\Sigma|^{\frac{n-1}{n}},
	$$
	where $r_0=\max\{d(o,x)|x\in \Sigma\}$, $\theta$ is the asymptotic volume ratio of $M$ given by \eqref{avr} and $b_0,b_1$ are defined in \eqref{b0} and \eqref{b1}. Furthermore, the equality holds if and only if $M$ is isometric to Euclidean space and $\Sigma$ is a flat ball.
\end{cor}

It is obvious that the above inequalities are nontrivial only when $\theta >0$.  We say that a complete Riemannian manifold with asymptotically nonnegative (Ricci) curvature has maximal volume growth if $\theta >0$.  Examples of such manifolds may be found in \cite{Ab1},  \cite{GH}, \cite{Calabi}, \cite{Kr1,Kr2}, and the first case of Theorem 1.2 in \cite{Un}, etc.

\section{The case of domains}	
Let $(M,g)$ be a complete noncompact $n$-dimensional Riemannian manifold of asymptotically nonnegative Ricci curvature with respect to a base point $o\in M$. Let $\Omega$ be a compact domain in $M$ with smooth boundary $\partial\Omega$ and $f$ be a smooth positive function on $\Omega$. Without loss of generality, we assume hereafter that $\Omega$ is connected.

By scaling, we may assume that
\begin{equation}\label{scale1}
	\int_{\partial\Omega} f+\int_\Omega|D f|+\int_\Omega2(n-1)b_1f=n\int_\Omega f^\frac{n}{n-1}.
\end{equation}
Due to \eqref{scale1} and the connectedness of $\Omega$, we can find a solution of the following Neumann boundary problem
\begin{equation}
	\left\{
	\begin{aligned}
		&\mathrm{div}(fD u)=nf^\frac{n}{n-1}-2(n-1)b_1f-|D f|,& \text{ in }\Omega,\\
		&\langle Du,\nu\rangle=1,
		& \text{ on }\partial\Omega,
	\end{aligned}
	\right.\label{neumann}
\end{equation}
where $\nu$ is the outward unit normal vector field along $\partial\Omega$. By standard elliptic regularity theory (see Theorem 6.31 in \cite{GT}), we know that $u\in C^{2,\gamma}$ for each $0<\gamma<1$.

As in \cite{Br2}, we set
$$
U:=\{x\in\Omega\setminus\partial\Omega:|Du(x)|<1\}.
$$
For any $r>0$, let 
$$
A_r=\{\bar{x}\in U: ru(x)+\frac{1}{2}{d}
(x,\exp_{\bar{x}}(rD u(\bar{x})))^2\geq
ru(\bar{x})+\frac{1}{2}r^2|D u(\bar{x})|^2,\ 
\forall x\in\Omega\}.
$$
Define a transport map $\Phi_r:\Omega\to M$ for each $r>0$ by
$$
\Phi_r(x)=\exp_x(rD u(x)), \quad\forall x\in\Omega.
$$
Using the regularity of the solution $u$ of the Neumann problem, we known that the map $\Phi_r$ is of class $C^{1,\gamma},0<\gamma<1$.

\begin{lem}\label{hessu}
	Assume that $x\in U$. Then we have
	\begin{equation*}
		\frac{1}{n}\Delta u\leq f^\frac{1}{n-1}-2\Big(\frac{n-1}{n}\Big)b_1.
	\end{equation*}
\end{lem}

\begin{proof}
	Using the Cauchy-Schwarz inequality and the property that $|Du|<1$ for $x\in U$, we get
	\begin{equation*}
		-\langle Df,Du\rangle\leq|Df|.
	\end{equation*}
	In terms of \eqref{neumann}, we derive that
	\begin{equation*}
		\begin{aligned}
			f\Delta u&=nf^\frac{n}{n-1}-2(n-1)b_1f-|Df|
			-\langle Df,Du\rangle\\
			&\leq nf^\frac{n}{n-1}-2(n-1)b_1f.
		\end{aligned}
	\end{equation*}
	This proves the assertion.
\end{proof}

The proofs of the following three lemmas are identical to those for Lemmas 2.2-2.4 in \cite{Br2} without any change for the case of asymptotically nonnegative Ricci curvature. So we omit them here.

\begin{lem}\label{tran}
	The set
	\[
	\{q\in M: d(x,q)<r,\  \forall x\in\Omega\}
	\]
	is contained in $\Phi_r(A_r)$.
\end{lem}

\begin{lem}
	Assume that $\bar{x}\in A_r$, and let 
	$\bar{\gamma}(t):=\exp_{\bar{x}}(tDu(\bar{x}))$ for all $t\in[0,r]$. 
	If $Z$ is a smooth vector field along $\bar{\gamma}$ satisfying
	$Z(r)=0$, then
	\[
	(D^2u)(Z(0),Z(0))+\int_0^r\big(|D_tZ(t)|^2-R(\bar{\gamma}'(t),Z(t),\bar{\gamma}'(t),Z(t))\big)\ dt\geq0.
	\]	
\end{lem}

\begin{lem}\label{jacovani}
	Assume that $\bar{x}\in A_r$, and let 
	$\bar{\gamma}(t):=\exp_{\bar{x}}(tD u(\bar{x}))$ for all $t\in[0,r]$. Moreover, 
	let $\{e_1,\dots,e_n\}$ be an orthonormal basis of $T_{\bar{x}}M$. Suppose that $W$ 
	is a Jacobi field along $\bar{\gamma}$ satisfying 	
	\[
	\langle D_tW(0),e_j\rangle=(D^2u)(W(0),e_j),\quad 1\leq j\leq n.
	\]
	If $W(\tau)=0$ for some $\tau\in(0,r)$, then $W$ vanishes identically.
\end{lem}

Now, we give two comparison results for later use.


\begin{lem}\label{comparison}
	Let $G$ be a continuous function on $[0,+\infty)$ and let $\phi,\psi\in C^2([0,+\infty))$ be solutions of the following problems
	$$
	\left\{
	\begin{aligned}
		&\phi''\leq {G}\phi ,\quad t\in(0,+\infty),\\
		&\phi(0)=1,\phi'(0)=b,
	\end{aligned} \right.
	\quad
	\left\{
	\begin{aligned}
		&\psi''\geq {G}\psi, \quad t\in(0,+\infty),\\
		&\psi(0)=1,\psi'(0)=\tilde{b},
	\end{aligned} \right.
	$$
	where $b,\tilde{b}$ are constants and $\tilde{b}\geq b$. If $\phi(t)>0$ for $t\in(0,T)$, then $\psi(t)>0$ in $(0,T)$ and
	$$
	\frac{\phi'}{\phi}\leq\frac{\psi'}{\psi}\quad\text{and}\quad\psi\geq\phi
	\quad\text{on }(0,T). 
	$$
\end{lem}

\begin{proof}
	Set $\beta=\sup\{t:\psi(t)>0\text{ in }(0,t)\}$ and $\tau = \min\{\beta,T\}$, so that $\phi$ and $\psi$ are both positive in $(0,\tau)$. The function $\psi'\phi-\psi\phi'$ is continuous on $[0,+\infty)$, nonnegative at $t=0$, and satisfies
	$$
	(\psi'\phi-\psi\phi')'=\psi''\phi-\psi\phi''\geq {G}(t)\psi\phi-
	{G}(t)\psi\phi=0,
	$$
	in $(0,\tau)$. Thus $\psi'\phi-\psi\phi'\geq 0$ on $[0,\tau)$, which implies
	\begin{equation}\label{2.3}
	\frac{\psi'}{\psi}\geq\frac{\phi'}{\phi}\quad\text{ in }[0,\tau).
	\end{equation}
	Integrating \eqref{2.3} between $0$ and $t$ $(0<t<\tau)$ yields
	$$
	\phi(t)\leq \psi(t),\quad\text{ in }[0,\tau).
	$$
	Since $\phi>0\text{ in }[0,\tau)$ by assumption, this forces $\tau=T$.
\end{proof}

\begin{lem}\label{h12}
	Let $G$ be a nonnegative continuous function on $[0,+\infty)$ satisfying \\
	$\int_0^{+\infty}G\ dt<+\infty$. Let $h_1,h_2\in C^2([0,+\infty))$ be solutions of the following problems
	\begin{equation}\label{2.4}
	\left\{
	\begin{aligned}
		&h_1''= {G}h_1 ,\quad t\in(0,+\infty),\\
		&h_1(0)=0,h_1'(0)=1,
	\end{aligned} \right.
	\quad
	\left\{
	\begin{aligned}
		&h_2''= {G}h_2, \quad t\in(0,+\infty),\\
		&h_2(0)=1,h_2'(0)=0.
	\end{aligned} \right.
	\end{equation}
	Then we have
	$$
	\lim_{t\to\infty}\frac{h_2}{h_1}=\lim_{t\to\infty}\frac{h'_2}{h'_1}
	\leq \int_0^{+\infty}G\ dt<\infty.
	$$
\end{lem}

\begin{proof}
	From \eqref{2.4}, we derive
	$$
	(h_2h_1'-h_1h_2')'(t)\equiv0,
	$$
	and thus
	\begin{equation}\label{hh1}
		(h_2h_1'-h_1h_2')(t)\equiv1
	\end{equation}
	in view of the initial values for $h_1$ and $h_2$. By derivation, one can find
	$$
	\Big(\frac{h_2}{h_1}\Big)'=\frac{h_2'h_1-h_1'h_2}{h_1^2}=\frac{-1}{h_1^2}<0,
	$$
	which implies that $\lim_{t\to+\infty}\frac{h_2(t)}{h_1(t)}$ exists. It is easy to show that 
	$$
	0\leq\Big(\frac{h'_2}{h'_1}\Big)'=
	\frac{G(h_2h_1'-h_1h_2')}{(h'_1)^2}
	\leq\frac{G}{(1+\int_0^tsG(s)ds)^2}
	\leq G,
	$$
	so we get 
	$$
	\frac{h'_2(t)}{h'_1(t)}\leq\int_0^{+\infty}G\ dt.
	$$ 
	 By Lemma 2.13 in \cite{PRS}, we have $h_1(t)\geq t$. Consequently, using \eqref{hh1} and $h'_1=1+\int_0^tGh_1 ds$, we obtain
	\begin{equation}\label{h123}
		\frac{h_2}{h_1}=\frac{h'_2}{h'_1}+\frac{1}{h_1h'_1}\leq\int_0^{+\infty}G\ dt+\frac{1}{t},\quad t\in(0,\infty).
	\end{equation}
	Letting $t\to\infty$, we have
	$$
	\begin{aligned}
		\lim_{t\to\infty}\frac{h_2}{h_1}=\lim_{t\to\infty}\frac{h'_2}{h'_1}
		\leq\int_0^{+\infty}G\ dt.
	\end{aligned}
	$$
\end{proof}

The next result is useful to study the growth of various balls on $M$ when their radii approach to infinity.

\begin{lem}\label{poly-degree-1}
	Let $h$ be the solution of \eqref{h}. Then
	$$
	\lim_{t\to+\infty}\frac{h(t-C)}{h(t)}=1 \text{ and }
	\lim_{t\to+\infty}\frac{h(tC)}{h(t)}=C,
	$$
	where $C$ is any positive constant.
\end{lem}

\begin{proof}
	From Lemma 2.13 in \cite{PRS}, we know $t\leq h(t)\leq e^{b_0}t$, and thus
	\begin{equation}\label{hp-leq}
		h'(t)=1+\int_0^t \lambda h\ dt\leq 1+b_0e^{b_0}.
	\end{equation}
	Clearly \eqref{hp-leq} means that $h'$ is nondecreasing and bounded from above. Consequently we have
	$$
	\lim_{t\to+\infty}\frac{h(t-C)}{h(t)}=
	\lim_{t\to+\infty}\frac{h'(t-C)}{h'(t)}=1
	$$
	and
	$$
	\lim_{t\to+\infty}\frac{h(tC)}{h(t)}=\lim_{t\to+\infty}\frac{Ch'(tC)}{h'(t)}=C.
	$$
	
\end{proof}

We are now turning to the proof of Theorem 1.1. 

\begin{proof}[\textbf{Proof of Theorem 1.1}]
	For any $r>0$ and $\bar{x}\in A_r$, let $\{e_1,\dots,e_n\}$ be an orthonormal basis of the tangent space $T_{\bar{x}}M$. Choosing the geodesic normal coordinates $(x^1,\dots,x^n)$ around $\bar{x}$, such that $\frac{\partial}{\partial x^i}=e_i$ at $\bar{x}$. Let $\bar{\gamma}(t):=\exp_{\bar{x}}(tD u(\bar{x}))$ for all $t\in[0,r]$. For $1\leq i\leq n$, let $E_i(t)$ be the parallel transport of $e_i$ along $\bar{\gamma}$. For $1\leq i\leq n$, let $X_i(t)$ be the Jacobi field along $\bar{\gamma}$ with the initial conditions of $X_i(0)=e_i$ and 
	$$
	\langle D_tX_i(0),e_j\rangle=(D^2u)(e_i,e_j),\quad1\leq j\leq n.
	$$
	Let $P(t)=(P_{ij}(t))$ be a matrix defined by
	$$
	P_{ij}(t)=\langle X_i(t),E_j(t)\rangle, \quad
	1\leq i,j\leq n.
	$$
	From Lemma \ref{jacovani}, we known $\det P(t)>0,\forall t\in[0,r)$. Obviously, $|\det D\Phi_t(\bar{x})|=\det P(t)>0$ for $t\in[0,r)$. Let $S(t)=(S_{ij}(t))$ be a matrix defined by
	$$
	S_{ij}(t)=R(\bar{\gamma}'(t),E_i(t),\bar{\gamma}'(t),E_j(t)),
	\quad
	1\leq i,j\leq n,
	$$
	where $R$ denotes the Riemannian curvature tensor of $M$. By the Jacobi equation, one can obtain
	\begin{equation}\label{jacobi}
		\left\{
		\begin{aligned}
			&P''(t)=-P(t)S(t),\quad t\in[0,r],\\
			&P_{ij}(0)=\delta_{ij},P_{ij}'(0)=(D^2u)(e_i,e_j).
		\end{aligned} \right.
	\end{equation}
	Let $Q(t)=P(t)^{-1}P'(t),t\in(0,r)$. Using \eqref{jacobi}, a simple computation  yields
	$$
	\frac{d}{dt}Q(t)=-S(t)-Q^2(t),
	$$
	where $Q(t)$ is symmetric. The assumption of asymptotically nonnegative Ricci curvature gives
	\begin{equation}\label{tr1}
		\begin{aligned}
			\frac{d}{dt}[\mathrm{tr} Q(t)]+\frac{1}{n}[\mathrm{tr}Q(t)]^2&\leq
			\frac{d}{dt}[\mathrm{tr} Q(t)]+\mathrm{tr}[Q^2(t)]\\
			&=-\mathrm{tr}S(t)\\
			&\leq (n-1)|D u(\bar{x})|^2\lambda(d(o,\bar{\gamma}(t))),
		\end{aligned}
	\end{equation}
	where $o$ is the base point. 
	Using triangle inequality and the definition of $A_r$, it is easy to see that
	\begin{equation}\label{tri}
		d(o,\bar{\gamma}(t))\geq\big|d(o,\bar{x})-d(\bar{x},\bar{\gamma}(t))\big|
		=\big|d(o,\bar{x})-t|Du(\bar{x})|\big|.
	\end{equation}
	Set 
	\begin{equation*}
		\begin{aligned}
			&g=\frac{1}{n}\mathrm{tr}Q,\\
			&\Lambda_{\bar{x}}(t)=\frac{(n-1)}{n}
			|D u(\bar{x})|^2
			\lambda(\big|d(o,\bar{x})-t|D u(\bar{x})|\big|).
		\end{aligned}
	\end{equation*}
	Noting that $\lambda$ is nonincreasing, it follows from \eqref{jacobi}, \eqref{tr1}, \eqref{tri} that
	$$
	\left\{
	\begin{aligned}
		&g'(t)+g(t)^2\leq \Lambda_{\bar{x}}(t), \quad t\in(0,r),\\
		&g(0)=\frac{1}{n}\Delta u(\bar{x}).
	\end{aligned} \right.
	$$
	If we take $\phi=e^{\int_0^tg(\tau)d\tau}$, then $\phi$ satisfies
	\begin{equation}\label{jac1}
		\left\{
		\begin{aligned}
			&\phi''\leq \Lambda_{\bar{x}}(t)\phi, \quad t\in(0,r),\\
			&\phi(0)=1,\phi'(0)=\frac{1}{n}\Delta u(\bar{x}).
		\end{aligned} \right.
	\end{equation}
	Next, we denote by $\psi_1,\psi_2$ the solutions of the following problems
	\begin{equation}\label{jac2}
		\left\{
		\begin{aligned}
			&\psi_1''= \Lambda_{\bar{x}}(t)\psi_1 ,\quad t\in(0,r),\\
			&\psi_1(0)=0,\psi_1'(0)=1,
		\end{aligned} \right.
		\quad
		\left\{
		\begin{aligned}
			&\psi_2''= \Lambda_{\bar{x}}(t)\psi_2 ,\quad t\in(0,r),\\
			&\psi_2(0)=1,\psi_2'(0)=0.
		\end{aligned} \right.
	\end{equation}
	Similar to the proof of \eqref{h123}, it is easy to verify that
	$$
	\frac{\psi_2}{\psi_1}(r)\leq \int_0^{+\infty}\Lambda_{\bar{x}}(t)\ dt+\frac{1}{r}
	\leq2\Big(\frac{n-1}{n}\Big)b_1|Du(\bar{x})|+\frac{1}{r}.
	$$
	Since $|Du(\bar{x})|<1$, we obtain
	\begin{equation}\label{d1}	
	\frac{\psi_2}{\psi_1}(r)\leq2\Big(\frac{n-1}{n}\Big)b_1+\frac{1}{r}.
	\end{equation}
Using Lemma 2.13 in \cite{PRS} and \eqref{jac2}, we deduce that
	\begin{equation}\label{comall}
		\begin{aligned}
			\psi_1(t)&\leq 
			\int_0^te^{\int_0^s \tau\Lambda_{\bar{x}}(\tau)d\tau}ds
			\\
			&\leq te^{\int_0^\infty \tau\Lambda_{\bar{x}}(\tau)d\tau}
			\\
			&=te^{\frac{n-1}{n}\int_0^\infty w\lambda(|d(o,\bar{x})-w|)dw}
			\\
			&\leq te^{\frac{n-1}{n}(2r_0b_1+b_0)},
		\end{aligned}
	\end{equation}
	where $r_0=\max\{d(o,x)|x\in\Omega\}$. 
	
	Let $\psi(t)=\psi_2(t)+\frac{1}{n}\Delta u(\bar{x})\psi_1(t)$. Using Lemma \ref{comparison}, one can get
	$$
	\frac{1}{n}\mathrm{tr}Q(t)=\frac{\phi'}{\phi}\leq\frac{\psi'}{\psi},\quad\forall t\in(0,r).
	$$
	Thus,
	\begin{equation}\label{d2}
		\frac{d}{dt}\log\det P(t)=\mathrm{tr}Q(t)\leq n\frac{\psi'}{\psi}.
	\end{equation}
	Consequently, \eqref{d2} implies
	$$
	|\det D\Phi_t(\bar{x})|=\det P(t)\leq \psi^n(t)=(\psi_2(t)+\frac{1}{n}\Delta u(\bar{x})\psi_1(t))^n
	$$
	for all $t\in[0,r]$. This gives
	\begin{align*}
		|\det D\Phi_r(\bar{x})|&\leq\Big(\frac{\psi_2(r)}{\psi_1(r)}+\frac{1}{n}\Delta u(\bar{x})\Big)^n\psi^n_1(r)
	\end{align*}
	for any $\bar{x}\in A_r$. Note that $0\leq\phi\leq\psi$. Using \eqref{d1}, \eqref{comall} and Lemma \ref{hessu}, we derive that
	\begin{equation}\label{d3}
		\begin{aligned}
			|\det D\Phi_r(\bar{x})|&\leq e^{(n-1)(2r_0b_1+b_0)}\Big(2\big(\frac{n-1}{n}\big)b_1+\frac{1}{r}+\frac{1}{n}\Delta u(\bar{x})\Big)^n
			r^n\\
			&\leq e^{(n-1)(2r_0b_1+b_0)}\Big(\frac{1}{r}+f^{\frac{1}{n-1}}(\bar{x})\Big)^n
			r^n
		\end{aligned}
	\end{equation}
	for any $\bar{x}\in A_r$. Moreover, by \eqref{h}, we obtain $h(t)\geq t$ and
	\begin{equation}\label{h1234}
		\lim_{t\to\infty}h'(t)=1+\int_0^\infty h(s)\lambda(s)\ ds\geq1+\int_0^\infty s\lambda(s)\ ds
		=1+b_0.
	\end{equation}
	Combining Lemma \ref{tran}, \eqref{d3} with the formula for change of variables in multiple integrals, we find that
	\begin{equation}\label{d4}
		\begin{aligned}
			&|\{q\in M:d(x,q)<r\text{ for all }x\in\Omega \}|\\
			&\leq\int_{A_r}|\det D\Phi_r|\\
			&\leq\int_{\Omega}e^{(n-1)(2r_0b_1+b_0)}(\frac{1}{r}+f^{\frac{1}{n-1}})^n
			r^n.
		\end{aligned}
	\end{equation}
	For $r>r_0$, the triangle inequality implies that
	\begin{equation}\label{inclusion}
		B_{r-r_0}(o)\subset\{q\in M:d(x,q)<r\text{ for all }x\in\Omega \}
		\subset B_{r+r_0}(o).
	\end{equation}
	From \eqref{avr}, \eqref{inclusion} and Lemma \ref{poly-degree-1}, it is easy to show that 
	\begin{equation}\label{avr-f-d}
		\begin{aligned}
			|B^n|\theta&=\lim_{r\to+\infty}\frac{B_{r-r_0}(o)}{n\int_0^{r-r_0}h(t)^{n-1}dt}
			\frac{\int_0^{r-r_0}h(t)^{n-1}dt}{\int_0^{r}h(t)^{n-1}dt}
			\\
			&\leq\lim_{r\to+\infty}\frac{|\{q\in M:d(x,q)<r\text{ for all }x\in\Omega \}|}
			{n\int_0^rh(t)^{n-1}dt}\\
			&\leq	\lim_{r\to+\infty}\frac{B_{r+r_0}(o)}{n\int_0^{r+r_0}h(t)^{n-1}dt}
			\frac{\int_0^{r+r_0}h(t)^{n-1}dt}{\int_0^{r}h(t)^{n-1}dt}	\\
			&=|B^n|\theta.
		\end{aligned}
	\end{equation}
	Dividing \eqref{d4} by $n\int_0^rh(t)^{n-1}dt$ and sending $r\to\infty$, it follows from \eqref{h1234} and \eqref{avr-f-d} that
	$$
	\begin{aligned}
		|B^n|\theta&\leq e^{(n-1)(2r_0b_1+b_0)}\int_\Omega f^{\frac{n}{n-1}}
		\lim_{r\to\infty}\frac{r^n}{n\int_0^rh(t)^{n-1}dt}
		\\
		&=e^{(n-1)(2r_0b_1+b_0)}\int_\Omega f^{\frac{n}{n-1}}\lim_{r\to\infty}\frac{1}{h'(t)^{n-1}}
		\\
		&\leq\Big(\frac{e^{2r_0b_1+b_0}}{1+b_0}\Big)^{n-1}\int_\Omega
		f^{\frac{n}{n-1}}.
	\end{aligned}
	$$
	Hence we obtain
	$$
	\int_{\partial \Omega} f+\int_\Omega |D f|+2(n-1)b_1\int_\Omega f
	\geq n|B^n|^{\frac 1n}
	\theta^{\frac1n}
	\Big(\frac{1+b_0}{e^{2r_0b_1+b_0}}\Big)^\frac{n-1}{n}
	\Big(\int_\Omega f^\frac{n}{n-1}\Big)^\frac{n-1}{n}.
	$$
\end{proof}

\begin{proof}[\textbf{Proof of Theorem 1.2}]
	Suppose the equality of Theorem \ref{thm1.1} holds. Then we have equalities in \eqref{d1} and \eqref{h1234} which force $\lambda\equiv0$. Thus $M$ has nonnegative Ricci curvature. The assertion follows immediately from Theorem 1.2 in \cite{Br2}.	
\end{proof}

\section{The case of submanifolds}	
In this section, we assume that the ambient space $M$ is a complete noncompact $(n+p)$-dimensional Riemannian manifold of asymptotically nonnegative sectional curvature with respect to a base point $o\in M$.  Let $\Sigma\subset M$ be a compact submanifold of dimension $n$ with or without boundary, and $f$ be a positive smooth function on $\Sigma$. Let $\bar{D}$ denote the Levi-Civita connection of $M$ and let $D^\Sigma$ denote the induced connection on $\Sigma$. The second fundamental form $B$ of $\Sigma$ is given by
$$
\langle B(X,Y),V\rangle=\langle \bar{D}_XY,V\rangle,
$$
where $X,Y$ are the tangent vector fields on $\Sigma$, $V$ is a normal vector field along $\Sigma$. The mean curvature vector of $\Sigma$ is defined by $H=\mathrm{tr}B$.

We only need to treat the case that $\Sigma$ is connected. By scaling, we can assume that 
\begin{equation}\label{scale2}
	\int_{\partial\Sigma}f+\int_\Sigma\sqrt{|D^\Sigma f|^2+f^2|H|^2}
	+2nb_1\int_\Sigma f=n\int_\Sigma f^{\frac{n}{n-1}}.
\end{equation}
By the connectedness of $\Sigma$ and \eqref{scale2}, there exists a solution of the following Neumann boundary problem
\begin{equation}\label{neus}
	\left\{
	\begin{aligned}
		&\text{div}_\Sigma(fD^\Sigma u)=nf^\frac{n}{n-1}-2nb_1f-\sqrt{|D^\Sigma f|^2+f^2|H|^2},& \text{ in }\Sigma,\\
		&\langle D^\Sigma u,\nu\rangle=1,
		& \text{ on }\partial\Sigma,
	\end{aligned}
	\right.
\end{equation}
where $\nu$ is the outward unit normal vector field of $\partial\Sigma$ with respect to $\Sigma$. Note that if $\partial\Sigma=\varnothing$, then the boundary condition in \eqref{neus} is void. By standard elliptic regularity theory (see Theorem 6.31 in \cite{GT}), we know that $u\in C^{2,\gamma}$ for each $0<\gamma<1$.

As in \cite{Br2}, we define
\begin{equation*}
	\begin{aligned}
		U:&=\{x\in\Sigma\setminus\partial\Sigma:|D^\Sigma u(x)|<1\},\\
		E:&=\{(x,y):x\in U,y\in T^\perp_x\Sigma,|D^\Sigma u(x)|^2+|y|^2<1\}.
	\end{aligned}
\end{equation*}
For each $r>0$, we denote by $A_r$ the set of all points $(\bar{x},\bar{y})\in E$ satisfying
$$
ru(x)+\frac{1}{2}{d}
(x,\exp_{\bar{x}}(rD^\Sigma u(\bar{x}))+r\bar{y})^2\geq
ru(\bar{x})+\frac{1}{2}r^2(|D^\Sigma u(\bar{x})|^2+|\bar{y}|^2)
$$
for all $x\in\Sigma$. Define the transport map $\Phi_r:T^\perp\Sigma\to M$ for each $r>0$ by
$$
\Phi_r(x,y)=\exp_x(rD^\Sigma u(x)+ry)
$$
for all $x\in\Sigma$ and $y\in T^\perp_x\Sigma$. The regularity of $u$ implies that $\Phi_r$ is of class $C^{1,\gamma}$, $0<\gamma<1$.

\begin{lem}\label{hess-sub}
	Assume that $(x,y)\in E$. Then we have
	$$
	\frac{1}{n}(\Delta_\Sigma u(x)-\langle H(x),y\rangle)\leq f^\frac{1}{n-1}(x)-2b_1.
	$$
\end{lem}

\begin{proof}
	Combining $|D^\Sigma u(x)|^2+|y|^2<1$ with Cauchy-Schwarz inequality, we obtain
	\begin{equation}\label{hessu-sub-1}
		\begin{aligned}
			&-\langle D^\Sigma f(x),D^\Sigma u(x)\rangle-f(x)\langle H(x),y\rangle\\
			&\leq\sqrt{|D^\Sigma f(x)|^2+f(x)^2|H(x)|^2}\sqrt{|D^\Sigma u(x)|^2+|y|^2}\\
			&\leq\sqrt{|D^\Sigma f(x)|^2+f(x)^2|H(x)|^2}.
		\end{aligned}
	\end{equation}
	In terms of \eqref{neus} and \eqref{hessu-sub-1}, one derives that
	$$
	\begin{aligned}
		&f(x)\Delta_\Sigma u(x)-f(x)\langle H(x),y\rangle\\
		&=nf(x)^{\frac{n}{n-1}}-2nb_1f-\sqrt{|D^\Sigma f(x)|^2+f(x)^2|H(x)|^2}\\
		&-\langle D^\Sigma f(x),D^\Sigma u(x)\rangle-f(x)\langle H(x),y\rangle\\
		&\leq nf(x)^{\frac{n}{n-1}}-2nb_1f.
	\end{aligned}
	$$
	The proof is completed.
\end{proof}

The following three lemmas are due to Brendle (Lemmas 4.2, 4.3, 4.5 in \cite{Br2}). Their proofs are independent of the curvature condition of ambient space too.

\begin{lem}\label{tran-sub}
	For each $0\leq\sigma<1$, the set
	$$
	\{q\in M:\sigma r<d(x,q)<r,\ \forall x\in\Sigma\}
	$$
	is contained in the set
	$$
	\Phi_r(\{(x,y)\in A_r:|D^\Sigma u(x)|^2+|y|^2>\sigma^2\}).
	$$
\end{lem}	

\begin{lem}
	Assume that $(\bar{x},\bar{y})\in A_r$, and let $\bar{\gamma}(t):=\exp_{\bar{x}}(tD^\Sigma u(\bar{x})+t\bar{y})$ for all $t\in[0,r]$. If $Z$ is a smooth vector field along $\bar{\gamma}$ satisfying $Z(0)\in T_{\bar{x}}\Sigma$ and $Z(r)=0$, then
	$$
	\begin{aligned}
		&((D^\Sigma)^2u)(Z(0),Z(0))-\langle B(Z(0),Z(0)),\bar{y}\rangle\\
		&+\int_0^r\big(|\bar{D}_tZ(t)|^2-\bar{R}(\bar{\gamma}'(t),Z(t),\bar{\gamma}'(t),Z(t))\big)dt\geq0.
	\end{aligned}
	$$
\end{lem}

\begin{lem}\label{vani-sub}
	Assume that $(\bar{x},\bar{y})\in A_r$, and let $\bar{\gamma}(t):=\exp_{\bar{x}}(tD^\Sigma u(\bar{x})+t\bar{y})$ for all $t\in[0,r]$. Let $\{e_1,\dots,e_n\}$ be an orthonormal basis of $T_{\bar{x}}\Sigma$. Suppose that $W$ is a Jacobi field along $\bar{\gamma}$ satisfying $W(0)\in T_{\bar{x}}\Sigma$ and $\langle\bar{D}_tW(0),e_j\rangle=((D^\Sigma)^2u)(W(0),e_j)-\langle B(W(0),e_j),\bar{y}\rangle$ for each $1\leq j\leq n$. If $W(\tau)=0$ for some $\tau\in(0,r)$, then $W$ vanishes identically.
\end{lem}

Now we begin the proof of Theorem \ref{thm1.4}.

\begin{proof}[\textbf{Proof of Theorem 1.4}]
For any $r>0$ and $(\bar{x},\bar{y})\in A_r$, let $\{e_i\}_{1\leq i\leq n}$ be any given orthonormal basis in $T_{\bar{x}}\Sigma$. Choose a normal coordinate system $(x^1,\cdots,x^n)$ on $\Sigma$ around $\bar{x}$ such that $\frac{\partial}{\partial x^i}=e_i$ at $\bar{x}\ (1\leq i\leq n)$. Let $\{e_\alpha\}_{n+1\leq \alpha\leq n+p}$ be an orthonormal frame field of $T^\perp\Sigma$ around $\bar{x}$ such that $\big((D^\Sigma)^\perp e_\alpha\big)_{\bar{x}}=0$ for $n+1\leq\alpha\leq n+p$, where $(D^\Sigma)^\perp$ denotes the normal connection in the normal bundle $T^\perp\Sigma$ of $\Sigma$. Any normal vector $y$ around $\bar{x}$ can be written as $y=\sum_{\alpha=n+1}^{n+p}y^\alpha e_\alpha$, and thus $(x^1,\cdots,x^n,y^{n+1},\cdots,y^{n+p})$ becomes a local coordinate system on the total space of the normal bundle $T^\perp\Sigma$.

Let $\bar{\gamma}(t):=\exp_{\bar{x}}(tD^\Sigma u(\bar{x})+t\bar{y})$ for all $t\in[0,r]$. For each $1\leq A\leq n+p$, we denote by $E_A(t)$ the parallel transport of $e_A(\bar{x})$ along $\bar{\gamma}$. For each $1\leq i\leq n$, let $X_i$ be the Jacobi field along $\bar{\gamma}$ with the following initial conditions
\begin{equation}\label{ja1}
	\begin{aligned}
		X_i(0)&=e_i,\\
		\langle\bar{D}_tX_i(0),e_j\rangle&=((D^\Sigma)^2u)(e_i,e_j)-\langle B(e_i,e_j),\bar{y}
		\rangle,\quad 1\leq j\leq n,\\
		\langle\bar{D}_tX_i(0),e_\beta\rangle&=
		\langle B(e_i,D^\Sigma u(\bar{x})),e_\beta
		\rangle,\quad n+1\leq \beta\leq n+p.
	\end{aligned}
\end{equation}
For each $n+1\leq\alpha\leq n+p$, let $X_\alpha$ be the Jacobi field along $\bar{\gamma}$ satisfying
\begin{equation}\label{ja2}
	X_\alpha(0)=0,\quad
	\bar{D}_tX_\alpha(0)=e_\alpha.
\end{equation}
Using Lemma $\ref{vani-sub}$, we known that $\{X_A(t)\}_{1\leq A\leq n+p}$ are linearly independent for each $t\in(0,r)$.

Let $P(t)=(P_{AB}(t))$ and $S(t)=(S_{AB}(t))$ be the matrices given by
$$
\begin{aligned}
	P_{AB}(t)&=\langle X_A(t),E_B(t)\rangle,\\
	S_{AB}(t)&=\bar{R}(\bar{\gamma}'(t),E_A(t),\bar{\gamma}'(t),E_B(t))
\end{aligned}
$$
for $1\leq A,B\leq n+p$ and $t\in[0,r]$, where $\bar{R}$ denotes the Riemannian curvature tensor of $M$.  Using the Jacobi equation and the initial conditions \eqref{ja1}, \eqref{ja2}, we have
\begin{equation}\label{jacobi-sub}
\begin{aligned}
	P''(t)&=-P(t)S(t),\\
	P_{AB}(0)&=\begin{bmatrix}
		\delta_{ij}&0\\
		0&0
	\end{bmatrix},
	\\
	P'_{AB}(0)&=\begin{bmatrix}
		((D^\Sigma)^2u)(e_i,e_j)-\langle B(e_i,e_j),\bar{y}
		\rangle&\langle B(e_i,D^\Sigma u(\bar{x})),e_\beta
		\rangle\\
		0&\delta_{\alpha\beta}
	\end{bmatrix}.
\end{aligned}
\end{equation}
Set $Q(t)=P(t)^{-1}P'(t),t\in(0,r)$. By \eqref{jacobi-sub}, a simple computation  yields
\begin{equation}\label{riccati-sub}
	\frac{d}{dt}Q(t)=-S(t)-Q^2(t),
\end{equation}
where $Q(t)$ is symmetric. For the matrices $P(t),Q(t)$, it is easy to derive their following asymptotic expansions (cf. \cite{Br2})
\begin{equation}\label{Q-asy}
\begin{aligned}
P(t)&=\begin{bmatrix}
	\delta_{ij}+O(t)&O(t)\\
	O(t) & t\delta_{\alpha\beta}+O(t^2)
\end{bmatrix},\\
Q(t)&=\begin{bmatrix}
	(D^\Sigma)^2u(e_i,e_j)-\langle B(e_i,e_j),\bar{y}
	\rangle+O(t)&O(1)
	\\
	O(1)&\frac{1}{t}\delta_{\alpha\beta}+O(1)
\end{bmatrix}
\end{aligned}
\end{equation}
as $t\to0^+$. In terms of \eqref{riccati-sub} and the curvature assumption for $M$, we deduce
\begin{equation}\label{tr1-sub}
\begin{aligned}
	&\frac{d}{dt}Q_{AA}(t)+Q_{AA}(t)^2\leq\frac{d}{dt}Q_{AA}(t)+\sum_{B=1}^{n+p}Q_{AB}Q_{BA}(t)
	\\
	&=-S_{AA}(t)
	\\
	&\leq(|D^\Sigma u(\bar{x})|^2+|\bar{y}|^2-\langle D^\Sigma u(\bar{x})+\bar{y},e_A\rangle^2)\lambda(d(o,\bar{\gamma}(t)))\\
	&\leq (|D^\Sigma u(\bar{x})|^2+|\bar{y}|^2-\langle D^\Sigma u(\bar{x})+\bar{y},e_A\rangle^2)\lambda(\big|d(o,\bar{x})-t|D^\Sigma u(\bar{x})+\bar{y}|\big|)
\end{aligned}
\end{equation}
for $1\leq A\leq n+p$, where the last inequality follows from the following triangle inequality
$$
	d(o,\bar{\gamma}(t))\geq\big|d(o,\bar{x})-d(\bar{x},\bar{\gamma}(t))\big|
	=\big|d(o,\bar{x})-t|D^\Sigma u(\bar{x})+\bar{y}|\big|.
$$
For $1\leq A\leq n+p$, we set 
$$
	\Lambda_{\bar{x},A}(t)=(|D^\Sigma u(\bar{x})|^2+|\bar{y}|^2-\langle D^\Sigma u(\bar{x})+\bar{y},e_A\rangle^2)\lambda(\big|d(o,\bar{x})-t|D^\Sigma u(\bar{x})+\bar{y}|\big|).
$$
Then we have
\begin{equation*}
\left\{
\begin{aligned}
	&Q'_{ii}(t)+Q_{ii}(t)^2\leq\Lambda_{\bar{x},i}(t),\quad t\in(0,r),
	\\
	&\lim_{t\to0^+}Q_{ii}(t)=\lambda_i,
\end{aligned}
\right.
\end{equation*}
where $\lambda_i=P'_{ii}(0)$.
Let $\phi_i$ be defined by
$$
\phi_i(t)=e^{\int_0^t Q_{ii}(\tau)d\tau}.
$$
Then $\phi_i$ satisfies 
\begin{equation}\label{Qi}
\left\{
\begin{aligned}
	&\phi_i''\leq \Lambda_{\bar{x},i}\phi_i, \quad t\in(0,r),
	\\
	&\phi_i(0)=1,\phi_i'(0)=\lambda_i.
\end{aligned} \right.
\end{equation}
Next, we denote by $\psi_{1,i},\psi_{2,i}$ solutions to the following problems
\begin{equation}\label{psi12}
\left\{
\begin{aligned}
	&\psi_{1,i}''= \Lambda_{\bar{x},i}\psi_{1,i} ,\quad t\in(0,r),
	\\
	&\psi_{1,i}(0)=0,\psi_{1,i}'(0)=1,
\end{aligned} \right.
\quad
\left\{
\begin{aligned}
	&\psi_{2,i}''= \Lambda_{\bar{x},i}\psi_{2,i} ,\quad t\in(0,r),
	\\
	&\psi_{2,i}(0)=1,\psi_{2,i}'(0)=0.
\end{aligned} \right.
\end{equation}
Similar to the proof of \eqref{h123}, \eqref{d1} and \eqref{comall}, we obtain
\begin{equation}\label{tan1}
\begin{aligned}
	\frac{\psi_{2,i}}{\psi_{1,i}}(r)&\leq \int_0^{+\infty}\Lambda_{\bar{x},i}(t)\ dt+\frac{1}{r}\\
	&\leq 2b_1{\frac{|D^\Sigma u(\bar{x})|^2+|\bar{y}|^2-\langle D^\Sigma u(\bar{x})+\bar{y},e_i\rangle^2}{\sqrt{|D^\Sigma u(\bar{x})|^2+\bar{y}^2}}}+\frac{1}{r}\\
	&\leq 2b_1{{\sqrt{|D^\Sigma u(\bar{x})|^2+\bar{y}^2}}}+\frac{1}{r}
\end{aligned}
\end{equation}
and 
\begin{equation}\label{tan2}
\psi_{1,i}(t)\leq te^{\frac{|D^\Sigma u(\bar{x})|^2+\bar{y}^2-\langle D^\Sigma u(\bar{x})+\bar{y},e_i\rangle^2}{|D^\Sigma u(\bar{x})|^2+\bar{y}^2}(2r_0b_1+b_0)},
\quad t\in(0,r),
\end{equation}
where $r_0=\max\{d(o,x)|x\in\Sigma\}$. Using Lemma \ref{comparison}, one can find from \eqref{Qi} and \eqref{psi12} that
\begin{equation}\label{tans}
Q_{ii}(t)=\frac{\phi'_i}{\phi_i}(t)\leq\frac{\psi'_{2,i}+\lambda_i\psi'_{1,i}}{\psi_{2,i}+\lambda_i\psi_{1,i}}(t).
\end{equation}
Similarly we obtain from \eqref{Q-asy} and \eqref{tr1-sub} that
$$
\left\{
\begin{aligned}
	&Q'_{\alpha\alpha}(t)+Q_{\alpha\alpha}(t)^2\leq\Lambda_{\bar{x},\alpha}(t),\quad t\in(0,r),\\
	&Q_{\alpha\alpha}(t)=\frac{1}{t}+O(1),\quad \text{as }t\to0^+
\end{aligned}\right.
$$
for $n+1\leq\alpha\leq n+p$. Set $\phi_\alpha(t)=te^{\int_0^t(Q_{\alpha\alpha}
	(\tau)-\frac{1}{\tau})d\tau}$. Then $\phi_\alpha$ satisfies
$$
\left\{
\begin{aligned}
	&\phi_\alpha''\leq \Lambda_{\bar{x},\alpha}\phi_\alpha, \quad t\in(0,r),\\
	&\phi_\alpha(0)=0,\phi_\alpha'(0)=1.
\end{aligned} \right.
$$
Next, we denote by $\psi_{1,\alpha}$ the unique solution to the following problem
\begin{equation}\label{psi1}
\left\{
\begin{aligned}
	&\psi_{1,\alpha}''= \Lambda_{\bar{x},\alpha}\psi_{1,\alpha} ,
	\quad t\in(0,r),\\
	&\psi_{1,\alpha}(0)=0,\psi_{1,\alpha}'(0)=1.
\end{aligned} \right.
\end{equation}
Similar to \eqref{comall}, we derive that
\begin{equation}\label{ver2}
\psi_{1,\alpha}\leq e^{\frac{|D^\Sigma u(\bar{x})|^2+\bar{y}^2-\langle D^\Sigma u(\bar{x})+\bar{y},e_\alpha\rangle^2}{|D^\Sigma u(\bar{x})|^2+\bar{y}^2}(2r_0b_1+b_0)}t,
\end{equation}
for $t\in(0,r)$. By Lemma 2.1 in \cite{PRS} we have
\begin{equation}\label{vers}
Q_{\alpha\alpha}(t)=\frac{\phi'_\alpha}{\phi_\alpha}(t)\leq 
\frac{\psi'_{1,\alpha}}{\psi_{1,\alpha}}(t).
\end{equation}
From \eqref{tans} and \eqref{vers}, it follows that
\begin{equation}\label{logdet}
	\frac{d}{dt}\log\det P(t)=\mathrm{tr}(Q(t))\leq
	\sum_i \frac{\psi'_{2,i}+\lambda_i\psi'_{1,i}}{\psi_{2,i}+\lambda_i\psi_{1,i}}(t)+\sum_\alpha \frac{\psi'_{1,\alpha}}{\psi_{1,\alpha}}(t).
\end{equation}
Combining \eqref{psi12}, \eqref{psi1} with the asymptotic properties in \eqref{Q-asy}, we conclude that
\begin{equation}\label{Ppsi-asy}
	\lim_{t\to0^+}\frac{\det P(t)}{\prod_i(\psi_{2,i}(t)+\lambda_i\psi_{1,i}(t))
		\prod_\alpha\psi_{1,\alpha}(t)}=1.
\end{equation}
Integrating \eqref{logdet} over $[\varepsilon,t]$ for $0<\varepsilon<t$ and using \eqref{Ppsi-asy} by letting $\varepsilon\to0^+$, it is easy to show that
$$
|\det \bar{D}\Phi_t(\bar{x},\bar{y})|=\det P(t)\leq \prod_i(\psi_{2,i}(t)+\lambda_i\psi_{1,i}(t))
\prod_\alpha\psi_{1,\alpha}(t).
$$
Note that $0\leq\phi_i\leq(\psi_{2,i}+\lambda_i\psi_{1,i})$ and $\psi_{1,i}\geq0\ (1\leq i\leq n)$. Combining \eqref{tan2}, \eqref{ver2} with arithmetric-geometric mean inequality, we obtain
$$
\begin{aligned}
	|\det \bar{D}\Phi_t(\bar{x},\bar{y})|
	&\leq\Big(\frac{1}{n}\sum_i\frac{\psi_{2,i}(t)}{\psi_{1,i}(t)}+\frac{1}{n}
	(\Delta_\Sigma u(\bar{x})-\langle H(\bar{x}),\bar{y}\rangle)\Big)^n\prod_A\psi_{1,A}(t)\\
	&\leq\Big(\frac{1}{n}\sum_i\frac{\psi_{2,i}(t)}{\psi_{1,i}(t)}+
	\frac{1}{n}(\Delta_\Sigma u(\bar{x})-\langle H(\bar{x}),\bar{y}\rangle)\Big)^nt^{n+p}e^{(n+p-1)(2r_0b_1+b_0)}
\end{aligned}
$$
which yields by \eqref{tan1} that
\begin{equation}\label{logsobo}
	\begin{aligned}
		&|\det \bar{D}\Phi_r(\bar{x},\bar{y})|\\
		&\leq
		(2b_1{{\sqrt{|Du(\bar{x})|^2+\bar{y}^2}}}+\frac{1}{r}+\frac{1}{n}
		(\Delta_\Sigma u(\bar{x})-\langle H(\bar{x}),\bar{y}\rangle))^nr^{n+p}e^{(n+p-1)(2r_0b_1+b_0)}
	\end{aligned}
\end{equation}
for all $(\bar{x},\bar{y})\in A_r$. Noting that ${\sqrt{|Du(\bar{x})|^2+\bar{y}^2}}<1$, we derive by Lemma \ref{hess-sub} and \eqref{logsobo} that
\begin{equation}\label{sub1}
	|\det \bar{D}\Phi_r(\bar{x},\bar{y})|
	\leq(\frac{1}{r}+f^\frac{1}{n-1}(\bar{x}))^n
	r^{n+p}e^{(n+p-1)(2r_0b_1+b_0)}
\end{equation}
for all $(\bar{x},\bar{y})\in A_r$.
Using Lemma \ref{tran-sub} and \eqref{sub1}, one may find in a similar way as the proof of Theorem 1.4 in \cite{Br2} that
\begin{equation}\label{br2-f}
\begin{aligned}
	&|\{p\in M:\sigma r<d(x,p)<r,\forall x\in\Sigma\}|\\
	&\leq\frac{p}{2}|B^p|(1-\sigma^2)e^{(n+p-1)(2r_0b_1+b_0)}\int_\Sigma
	(\frac{1}{r}+f^\frac{1}{n-1}(\bar{x}))^nr^{n+p},
\end{aligned}
\end{equation}
for all $r>0$ and all $0\leq\sigma<1$. 
Similar to the proof of \eqref{avr-f-d}, one can obtain by using Lemma \ref{poly-degree-1} that
\begin{equation}\label{avr-f-s}
\begin{aligned}
	&\lim_{r\to+\infty}\frac{|\{p\in M:\sigma r<d(x,p)<r,\forall x\in\Sigma\}|}{(n+p)\int_0^rh^{n+p-1}dt}\\
	&=|B^{n+p}|\theta\lim_{r\to+\infty}
	(1-\sigma\frac{h^{n+p-1}(\sigma r)}{h^{n+p-1}(r)})\\
	&=|B^{n+p}|(1-\sigma^{n+p})\theta.
\end{aligned}
\end{equation}
Dividing \eqref{br2-f} by $(n+p)\int_0^rh(t)^{n+p-1}dt$ and sending $r\to+\infty$, we deduce by using \eqref{h1234} and \eqref{avr-f-s} that
\begin{equation}\label{ddl1-f}
\begin{aligned}
	&=|B^{n+p}|(1-\sigma^{n+p})\theta\\
	&\leq
	\frac{p}{2}|B^p|(1-\sigma^2)e^{(n+p-1)(2r_0b_1+b_0)}\int_\Sigma f
	^{\frac{n}{n-1}}
	\lim_{r\to+\infty}\frac{r^{n+p}}{(n+p)\int_0^rh(t)^{n+p-1}dt}
	\\
	&\leq\frac{p}{2}|B^p|(1-\sigma^2)\Big(\frac{e^{2r_0b_1+b_0}}{1+b_0}\Big)^{n+p-1}
	\int_\Sigma f
	^{\frac{n}{n-1}}.
\end{aligned}
\end{equation}
for all $0\leq\sigma<1$. Now, if we divide \eqref{ddl1-f} by $1-\sigma$ and let $\sigma\to1$, we have
\begin{equation}\label{dll-ff}
	(n+p)|B^{n+p}|\theta\leq p|B^p|\Big(\frac{e^{2r_0b_1+b_0}}{1+b_0}\Big)^{n+p-1}
	\int_\Sigma f
	^{\frac{n}{n-1}}.
\end{equation}
Hence \eqref{scale2} and \eqref{dll-ff} imply that
$$
\begin{aligned}
	&\int_{\partial\Sigma}f+\int_\Sigma\sqrt{|D^\Sigma f|^2+f^2|H|^2}
	+2nb_1\int_\Sigma f\\
	&\geq
	n\Big(\frac{(n+p)|B^{n+p}|}{p|B^p|}\Big)^{\frac{1}{n}}\theta^{\frac{1}{n}}
	\Big(\frac{1+b_0}{e^{2r_0b_1+b_0}}\Big)^{\frac{n+p-1}{n}}
	\Big(\int_\Sigma f^{\frac{n}{n-1}}\Big) ^{\frac{n-1}{n}}.
\end{aligned}
$$
\end{proof}

\begin{proof}[\textbf{Proof of Theorem 1.6}]
Suppose the equality of Theorem \ref{thm1.4} holds. Then we have equality in both \eqref{h1234} and \eqref{tan1} and either one forces $\lambda\equiv0$. Thus $M$ has nonnegative sectional curvature. The assertion follows immediately from Theorem 1.6 in \cite{Br2}.	
\end{proof}

Finally we would like to mention that we have established a Sobolev type inequality for manifolds with density and asymptotically nonnegative Bakery-\'Emery Ricci curvature in \cite{DLL2} and a logarithmic Sobolev type inequality for closed submanifolds in manifolds with asymptotically nonnegative sectional curvature in \cite{DLL3}.

\section*{Acknowledgements}
We would like to thank Prof. J. Z. Zhou for drawing our attention to Brendle's paper. Thanks also due to Prof. M. Fogagnolo for his useful communication.

\bibliographystyle{plain}
\bibliography{DLLsobo}

\noindent\mbox{Yuxin Dong and Lingen Lu} \\
\mbox{School of Mathematical Sciences}\\
\mbox{220 Handan Road, Yangpu District}\\
\mbox{Fudan University}\\
\mbox{Shanghai, 20043}\\
\mbox{P.R. China}\\
\mbox{\textcolor{blue}{yxdong@fudan.edu.cn}}\\
\mbox{\textcolor{blue}{19110180021@fudan.edu.cn}}

\newpage

\noindent\mbox{Hezi Lin}\\
\mbox{School of Mathematics and Statistics \&  FJKLMAA}\\
\mbox{Fujian Normal University}\\
\mbox{Fuzhou,  350108}\\
\mbox{P.R. China}\\
\mbox{\textcolor{blue}{ lhz1@fjnu.edu.cn}}

\end{document}